\documentclass[leqno,11pt]{article}
\usepackage[utf8x]{inputenc}
\usepackage{microtype}
\usepackage{amsmath}
\usepackage{amsthm}
\usepackage{amssymb}
\usepackage{mathrsfs} 
\usepackage[T1]{fontenc}
\usepackage{ae,aecompl}
\usepackage{times}
\usepackage[english]{babel}
\usepackage{aliascnt}
\usepackage[colorlinks,pagebackref,hypertexnames=false]{hyperref} \usepackage[alphabetic,backrefs,msc-links]{amsrefs}
\usepackage{bookmark}
\usepackage{esint} 
\usepackage{titlesec}

\usepackage[yyyymmdd]{datetime}

\usepackage{titlesec}

\numberwithin{equation}{section}

\newcommand{\mathsc}[1]{{\normalfont\textsc{#1}}}
\newcommand{\scN}{{\mathsc{n}}}

\newcommand{\N}{\bN}
\newcommand{\Z}{\bZ}

\newcommand{\R}{\bR}
\newcommand{\C}{\bC}

\newcommand{\su}{\mathfrak{su}}

\newcommand{\Spin}{{\rm Spin}}
\newcommand{\SU}{{\rm SU}}

\newcommand{\SL}{\mathrm{SL}}
\newcommand{\U}{{\rm U}}

\newcommand{\PSL}{\bP\SL}

\DeclareMathOperator{\Hom}{Hom}

\DeclareMathOperator{\rk}{rk}

\DeclareMathOperator{\coim}{coim}

\DeclareMathOperator{\tr}{tr}

\renewcommand{\det}{\operatorname{det}}

\newcommand{\del}{\partial}

\newcommand{\id}{{\rm id}}
\newcommand{\loc}{{\rm loc}}

\renewcommand{\epsilon}{\varepsilon}

\def\({\left(}
\def\){\right)}
\def\<{\left\langle}
\def\>{\right\rangle}

\newcommand{\co}{\mskip0.5mu\colon\thinspace}
\newcommand{\floor}[1]{\lfloor#1\rfloor}

\newcommand{\into}{\hookrightarrow}
\newcommand{\iso}{\cong}

\newcommand{\qand}{\quad\text{and}}
\newcommand{\andq}{\text{and}\quad}
\newcommand{\qandq}{\quad\text{and}\quad}

\usepackage{mathtools}

\DeclarePairedDelimiter{\set}{\lbrace}{\rbrace}

\newcommand{\hkred}{{/\!\! /\!\! /}}

\makeatletter
\renewcommand\xleftrightarrow[2][]{%
  \ext@arrow 9999{\longleftrightarrowfill@}{#1}{#2}}
\newcommand\longleftrightarrowfill@{%
  \arrowfill@\leftarrow\relbar\rightarrow}
\makeatother

\newcommand{\rd}{{\rm d}}

\newcommand{\rG}{{\rm G}}

\newcommand{\bC}{\mathbf{C}}

\newcommand{\bN}{\mathbf{N}}

\newcommand{\bP}{\mathbf{P}}

\newcommand{\bR}{\mathbf{R}}

\newcommand{\bZ}{\mathbf{Z}}

\newcommand{\sA}{\mathscr{A}}

\newcommand{\sK}{\mathscr{K}}
\newcommand{\sL}{\mathscr{L}}

\newcommand{\sP}{\mathscr{P}}

\newcommand{\sW}{\mathscr{W}}

\newcommand{\fe}{{\mathfrak e}}

\newcommand{\fg}{{\mathfrak g}}

\newcommand{\fl}{{\mathfrak l}}

\newcommand{\fs}{{\mathfrak s}}

\newcommand{\fF}{{\mathfrak F}}

\newcommand{\fI}{{\mathfrak I}}

\newcommand{\fM}{{\mathfrak M}}

\usepackage{slashed}
\newcommand{\slD}{\slashed D}
\newcommand{\slS}{\slashed S}

\renewcommand{\eqref}[1]{\hyperref[#1]{\rm(\ref*{#1})}}

\def\makeautorefname#1#2{\AtBeginDocument{\expandafter\def\csname#1autorefname\endcsname{#2}}}

\newcommand{\mynewtheorem}[2]{
  \newaliascnt{#1}{equation}          
  \newtheorem{#1}[#1]{#2}
  \aliascntresetthe{#1}
  \makeautorefname{#1}{#2}
}

\newcommand{\mynewproblem}[2]{
  \newaliascnt{#1}{myProblem}          
  \newtheorem{#1}[#1]{#2}
  \aliascntresetthe{#1}
  \makeautorefname{#1}{#2}
}

\makeautorefname{table}{Table}

\mynewtheorem{theorem}{Theorem}
\newtheorem*{theorem*}{Theorem}
\mynewtheorem{prop}{Proposition}
\mynewtheorem{cor}{Corollary}
\mynewtheorem{construction}{Construction}
\mynewtheorem{lemma}{Lemma}
\mynewtheorem{conjecture}{Conjecture}
\mynewtheorem{hyp}{Hypothesis}
\newtheorem{step}{Step}

\numberwithin{substep}{step}
\makeautorefname{step}{Step}
\makeautorefname{substep}{Step}

\newtheorem{case}{Case}

\numberwithin{subcase}{case}
\makeautorefname{case}{Case}
\makeautorefname{subcase}{case}

\theoremstyle{remark}
\mynewtheorem{remark}{Remark}
\newtheorem*{remark*}{Remark}

\theoremstyle{definition}
\mynewtheorem{definition}{Definition}
\mynewtheorem{example}{Example}
\mynewtheorem{exercise}{Exercise}
\mynewproblem{problem}{Problem}
\mynewtheorem{solution}{Solution}
\mynewtheorem{convention}{Convention}
\newtheorem*{convention*}{Convention}
\newtheorem*{conventions*}{Conventions}
\mynewtheorem{question}{Question}
        
\makeautorefname{chapter}{Chapter}
\makeautorefname{section}{Section}
\makeautorefname{subsection}{Section}
\makeautorefname{subsubsection}{Section}
\makeautorefname{footnote}{Footnote}

\usepackage{enumerate}

\author{
  Andriy Haydys \\
  Universität Bielefeld
  \and
  Thomas Walpuski \\
  Massachusetts Institute of Technology
}
\title{
  A compactness theorem for the Seiberg--Witten equation with multiple spinors in dimension three
}
\date{2016-07-07}

\begin{document}
\maketitle

\begin{abstract}
  We prove that a sequence of solutions of the Seiberg--Witten equation with multiple spinors in dimension three can degenerate only by converging (after rescaling) to a Fueter section of a bundle of moduli spaces of ASD instantons.
\end{abstract}

\paragraph{Changes to the published version}
This is an update of our article published as \href{http://dx.doi.org/10.1007/s00039-015-0346-3}{Geometric and Functional Analysis 25 (2015), no. 6, 1799–1821}.
The present version corrects a mistake in \autoref{Prop_nSW0Fueter} pointed out to us by Aleksander Doan,
namely the connection $A$ does not need to be flat if $n\ge 3$. 
This is because the canonical connection on $\mu^{-1}(0) \to \mathring{M}_{1,n}$ is flat only if $n = 2$, whereas we claimed this to be true for all $n$.
This was used to deduce that the limit connection $A$ in \autoref{Thm_A} is flat with $\Z_2$--monodromy.

\section{Introduction}
\label{Sec_Introduction}

Let $M$ be an oriented Riemannian closed three--manifold.
Fix a $\Spin$--structure $\fs$ on $M$ and denote by $\slS$ the associated spinor bundle;
also fix a $\U(1)$--bundle $\sL$ over $M$, a positive integer $n \in \N$ and a $\SU(n)$--bundle $E$ together with a connection $B$.
We consider pairs $(A,\Psi) \in \sA(\sL) \times \Gamma\(\Hom(E,\slS \otimes \sL)\)$ consisting of a connection $A$ on $\sL$ and an $n$--tuple of twisted spinors $\Psi$ satisfying the \emph{Seiberg--Witten equation with $n$ spinors}:
\begin{align}
  \label{Eq_nSW0}
  \begin{split}
    \slD_{A\otimes B} \Psi &= 0 \qand \\
    F_A &= \mu(\Psi).
  \end{split}
\end{align}
Here $\mu\co \Hom(E,\slS \otimes \sL) \to \fg_{\sL}\otimes \su(\slS) = i \su(\slS)$ is defined by
\begin{equation}
  \label{Eq_Mu}
  \mu(\Psi) := \Psi\Psi^* - \frac12|\Psi|^2\,\id_\slS
\end{equation}
and we identify $\Lambda^2 T^*\!M$ with $\su(\slS)$ via
\begin{equation}
  \label{Eq_Lambda=su}
  e^i\wedge e^j \mapsto \frac12[\gamma(e^i),\gamma(e^j)] = \epsilon_{ijk} \gamma(e^k).
\end{equation}

If $n=1$, then $E$ and $B$ are trivial, since $\SU(1) = \set{1}$, and \eqref{Eq_nSW0} is nothing but the classical Seiberg--Witten equation in dimension three, which has been studied with remarkable success, see, e.g., \cites{Chen1997,Lim2000,Kronheimer2007}. 
A key ingredient in the analysis of \eqref{Eq_nSW0} with $n=1$ is the identity
\begin{equation*}
  \<\mu(\Psi)\Psi,\Psi\> = \frac12|\Psi|^4,
\end{equation*}
which combined with the Weitzenböck formula yields an a priori bound on $\Psi$ and, therefore, immediately gives compactness of the moduli spaces of solutions to \eqref{Eq_nSW0}.
After taking care of issues to do with transversality and reducibles,
counting solutions of \eqref{Eq_nSW0} leads to an invariant of three--manifolds. 

The above identity does not hold for $n \geq 2$ and, more importantly, $\mu$ is no longer proper; hence, the $L^2$--norm of $\Psi$ is not bounded a priori.
From an analytical perspective the difficult case is when this $L^2$--norm becomes very large;
however, also the case of very small $L^2$--norm deserves special attention as it corresponds to reducible solutions of \eqref{Eq_nSW0}.
With this in mind it is natural to blow-up \eqref{Eq_nSW0}, that is, to consider triples $(A,\Psi,\alpha) \in \sA(\sL) \times \Gamma\(\Hom(E,\slS\otimes \sL)\) \times [0,\pi/2]$ satisfying 
\begin{equation}
  \label{Eq_nSW}
  \begin{split}
    \|\Psi\|_{L^2} &= 1, \\
    \slD_{A\otimes B} \Psi &= 0 \qand \\
    \sin(\alpha)^2 F_A &= \cos(\alpha)^2 \mu(\Psi),
  \end{split}  
\end{equation}
c.f.~\cite{Kronheimer2007}*{Section 2.5}.
The difficulty in the analysis can now be understood as follows:
for $\alpha \in (0,\pi/2]$ equation \eqref{Eq_nSW} is elliptic (after gauge fixing), but as $\alpha$ tends to zero it becomes degenerate.

The following is the main result of this article:
\begin{theorem}
  \label{Thm_A}
  Let $(A_i,\Psi_i,\alpha_i) \in \sA(\sL) \times \Gamma\(\Hom(E,\slS\otimes \sL)\) \times (0,\pi/2]$ be a sequence of solutions of \eqref{Eq_nSW}.
  If $\limsup \alpha_i > 0$, then after passing to a subsequence and up to gauge transformations $(A_i,\Psi_i,\alpha_i)$ converges smoothly to a limit $(A,\Psi,\alpha)$.
  If $\limsup \alpha_i = 0$, then after passing to a subsequence the following holds:
  \begin{itemize}
    \item
      There is a closed nowhere-dense subset $Z \subset M$, a connection $A$ on $\sL|_{M\setminus Z}$ and $\Psi \in \Gamma\(M\setminus Z, \Hom(E,\slS\otimes \sL)\)$ such that $(A, \Psi, 0)$ solves \eqref{Eq_nSW}.
      $|\Psi|$ extends to a Hölder continuous function on all of $M$ and $Z = |\Psi|^{-1}(0)$. 
    \item
      On $M \setminus Z$, up to gauge transformations, $A_i$ converges weakly in $W^{1,2}_\loc$ to $A$ and $\Psi_i$ converges weakly in $W^{2,2}_\loc$ to $\Psi$.
      There is a constant $\gamma > 0$ such that $|\Psi_i|$ converges to $|\Psi|$ in $C^{0,\gamma}$ on all of $M$.
  \end{itemize}
\end{theorem}

\begin{remark*}
  \autoref{Prop_nSW0FueterConnection} gives more detailed information about the limit $(A, \Psi,0)$.
  In particular, if $n = 2$, then $A$ is flat with monodromy in $\Z_2$. 
\end{remark*}

\begin{remark}
  \autoref{Thm_A} should be compared with the results of Taubes on $\PSL(2, \C)$--connections on three--manifolds with curvature bounded in $L^2$~\cite{Taubes2012}*{Theorem~1.1}.
  Our proof heavily relies on his insights and techniques.
\end{remark}

\begin{remark}
  Taubes' very recent work~\cite{Taubes2014}*{Theorems 1.2, 1.3, 1,4 and 1.5} implies detailed regularity properties for $Z$;
  in particular, $Z$ has Hausdorff dimension at most one.
  To see that his theorems apply in our situation note that $Z$ is the zero locus of a $\Z_2$ harmonic spinor by \autoref{App_Fueter}.
\end{remark}

As is discussed in \autoref{App_Fueter}, gauge equivalence classes of nowhere-vanishing solutions of \eqref{Eq_nSW} with $\alpha = 0$ correspond to Fueter sections of a bundle $\fM$ with fibre $\mathring M_{1,n}$, the framed moduli space of centred charge one $\SU(n)$ ASD instantons on $\R^4$.
In particular, while \eqref{Eq_nSW} degenerates as $\alpha$ tends to zero, for $\alpha = 0$ it is equivalent to an elliptic partial differential equation, away from the zero-locus of $\Psi$.
Morally, this is why one can hope to prove \autoref{Thm_A}.

In view of \autoref{Thm_A}, the count of solutions of \eqref{Eq_nSW} can depend on the choice of (generic) parameters in $\sP$ (the space of metrics on $M$ and connections on $E$):
since $\mathring M_{1,n}$ is a cone and the Fueter equation has index zero, one expects Fueter sections of $\fM$ to appear (only) in codimension one;
thus, the count of solutions of \eqref{Eq_nSW} can jump along a path of parameters in $\sP$.
In other words:
there is a set $\sW \subset \sP$ of codimension one and the number of solutions of \eqref{Eq_nSW} depends on the connected component of $\sP\setminus\sW$.
In the study of gauge theory on $\rG_2$--manifolds the count of $\rG_2$--instantons also undergoes a jump whenever a solution of \eqref{Eq_nSW} with $\alpha=0$ appears, with $M$ an associative submanifold of a $\rG_2$--manifold and $B$ the restriction of a $\rG_2$--instanton to $M$, see~\cites{Donaldson2009,Walpuski2013a,Walpuski2013}.
So while both the count of $\rG_2$--instantons and the count of solutions of \eqref{Eq_nSW} cannot be invariants, there is hope that a suitable combination of counts of $\rG_2$--instantons and solutions of (generalisations of) \eqref{Eq_nSW} on associative submanifolds will yield an invariant of $\rG_2$--manifolds.
We will discuss this circle of ideas in more detail elsewhere.

\paragraph{Outline of the proof of \autoref{Thm_A}}
The Weitzenböck formula leads to a priori bounds which directly prove the first half of \autoref{Thm_A}.
The proof of the second half is more involved.
For a solution $(A,\Psi,\alpha)$ of \eqref{Eq_nSW}, we show that the (renormalised) $W^{2,2}_A$--norm of $\Psi$ on a ball $B_r(x)$ is uniformly bounded provided the radius is smaller than the \emph{critical radius}
\begin{equation*}
  \rho = \sup \set{ r : r^{1/2}\|F_A\|_{L^2(B_r(x))} \leq 1 }.
\end{equation*}
To control $\rho$ we use a \emph{frequency function} $\scN(r)$, which---roughly speaking---measures the vanishing order of $\Psi$ near $x$.
More precisely, we prove that there exists a constant $\omega > 0$, depending only on the geometry of $M$, such that $\scN(50r)\leq \omega$ implies $\rho\ge r$.
We also show that for any $\omega,\epsilon > 0$ there exists $r>0$ such that $\scN(r) \leq \omega$ provided $|\Psi|(x) \geq \epsilon$.
Thus, we can establish convergence outside the subset $Z = \set{ x \in M : \limsup |\Psi_i|(x)=0 }$.

\paragraph{Acknowledgements}

We would like to thank Cliff Taubes for helpful remarks on \cite{Taubes2012} and an anonymous referee for pointing out a mistake in an earlier version of this article.
The first named author was supported by the German Research Foundation (DFG) within the CRC 701. 
The second named author was supported by European Research Council Grant 247331.
A large part of the work presented in this article has been carried out while the second named author was a postdoc at Imperial College London.

\begin{conventions*}
  We write $x \lesssim y$ (or $y \gtrsim x$) for $x \leq c y$ with $c > 0$ a universal constant, which depends only on the geometry of $M$, $E$ and $B$;
  should $c$ depend on further data we indicate that by a subscript.
  $O(x)$ denotes a quantity $y$ with $|y| \lesssim x$.
  We denote by $r_0$ a constant $0 < r_0 \ll 1$; in particular, $r_0 \leq {\rm injrad}(M)$.
  We assume that all radii $r$ on $M$ under consideration are less than $r_0$.
  Throughout the rest of this article $\sL$, $E$ and $B$ are fixed.
\end{conventions*}


\section{A priori estimates}
\label{Sec_APrioriEstimates}

In this section we prove the following a priori estimates:

\begin{prop}
  \label{Prop_APrioriBounds}
  Every solution $(A,\Psi,\alpha) \in \sA(\sL) \times \Gamma\(\Hom(E,\slS\otimes \sL)\) \times (0,\pi/2]$ of \eqref{Eq_nSW} satisfies
  \begin{equation*}
    \|\Psi\|_{L^\infty}
    = O(1)
  \end{equation*}
  and, for each $x\in M$ and $r > 0$,
  \begin{align*}
    \|\nabla_{A\otimes B}\Psi\|_{L^2(B_r(x))}
    &= O(r^{1/2}) \qand \\
    \|\mu(\Psi)\|_{L^2(B_r(x))}
    &= O(r^{1/2}\tan(\alpha)).
  \end{align*}
\end{prop}

This implies the first part of \autoref{Thm_A} because if $\limsup \alpha_i >0$, then \eqref{Eq_nSW} does not degenerate and standard methods apply:

\begin{prop}
  In the situation of \autoref{Thm_A} if $\limsup \alpha_i >0$, then, after passing to a subsequence and up to gauge transformations, $(A_i,\Psi_i,\alpha_i)$ converges in $C^\infty$ to a limit $(A,\Psi,\alpha)$ solving \eqref{Eq_nSW}.
\end{prop}

By the Banach--Alaoglu theorem we have the following proposition:

\begin{prop}
  \label{Prop_NormConvergence}
  In the situation of \autoref{Thm_A} after passing to a subsequence $|\Psi_i|$ converges weakly in $W^{1,2}$ to a bounded limit $|\Psi|$.
\end{prop}

\begin{remark}
  Note that we have not yet constructed $\Psi$;
  however, we will show later that the notation $|\Psi|$ is indeed justified.
\end{remark}

The key to proving \autoref{Prop_APrioriBounds} are the Weitzenböck formula~\eqref{Eq_Weitzenbock}, the algebraic identity~\eqref{Eq_Mu1} and the integration by parts formula~\eqref{Eq_IntegrationByParts}.

\begin{prop}
  \label{Prop_Weitzenbock}
  For all $(A,\Psi) \in \sA(\sL) \times \Gamma\(\Hom(E,\slS\otimes \sL)\)$
  \begin{equation}
    \label{Eq_Weitzenbock}
    \slD_{A\otimes B}^*\slD_{A\otimes B} \Psi
    = \nabla_{A\otimes B}^*\nabla_{A \otimes B} \Psi + \frac{s}{4}\Psi + F_A\Psi + F_B\Psi
  \end{equation}
  with $s$ denoting the scalar curvature of $g$ and $F_A$ and $F_B$ acting via the isomorphism defined in \eqref{Eq_Lambda=su}.
  \qed
\end{prop}

\begin{prop}
  For all $\Psi \in \Gamma\(\Hom(E,\slS\otimes\sL)\)$
  \begin{equation}
    \label{Eq_Mu1}
    \<\mu(\Psi)\Psi,\Psi\>
    = |\mu(\Psi)|^2.
  \end{equation}
\end{prop}

\begin{proof}
  This follows from a simple computation:
  \begin{equation*}
    \<\mu(\Psi)\Psi,\Psi\>
    = \<\mu(\Psi),\Psi\Psi^*\>
    = \langle\mu(\Psi),\Psi\Psi^*-\frac12|\Psi|^2\,\id_\slS\rangle
    = |\mu(\Psi)|^2.
    \qedhere
  \end{equation*}
\end{proof}

\begin{prop}
  \label{Prop_IntegrationByParts}
  Suppose $(A,\Psi,\alpha) \in \sA(\sL) \times \Gamma\(\Hom(E,\slS\otimes \sL)\) \times (0,\pi/2]$ satisfies
  \begin{equation}
    \label{Eq_nSWNotNormalised}
    \begin{split}
      \slD_{A\otimes B} \Psi &= 0 \qand \\
      \sin(\alpha)^2 F_A &= \cos(\alpha)^2 \mu(\Psi).
    \end{split}
  \end{equation}
  If $f$ is any smooth function on $M$ and $U$ is a closed subset of $M$ with smooth boundary, then
  \begin{multline}
    \label{Eq_IntegrationByParts}
    \int_U \Delta f \cdot |\Psi|^2 + f \cdot \(\frac{s}{2}|\Psi|^2 + 2\<F_B\Psi,\Psi\> +  2\tan(\alpha)^{-2} |\mu(\Psi)|^2 + 2|\nabla_A\Psi|^2\) \\ = \int_{\del U} f \cdot \del_\nu |\Psi|^2 -\del_\nu f \cdot |\Psi|^2.
  \end{multline}
  Here $\nu$ denotes the outward pointing normal vector field.
\end{prop}

\begin{proof}
  Combine \eqref{Eq_nSW}, \eqref{Eq_Weitzenbock} and \eqref{Eq_Mu1} to obtain
  \begin{equation}
    \label{Eq_DeltaPsi}
    \frac12\Delta|\Psi|^2 + \frac{s}{4}|\Psi|^2 + \<F_B\Psi,\Psi\> + \tan(\alpha)^{-2}|\mu(\Psi)|^2 + |\nabla_A\Psi|^2
    = 0.
  \end{equation}
  The identity \eqref{Eq_IntegrationByParts} now follows from
  \begin{equation*}
    \int_U \Delta f \cdot g - f \cdot \Delta g
    = \int_{\del U} f \cdot \del_\nu g - \del_\nu f \cdot g 
  \end{equation*}
  with $g = |\Psi|^2$.
\end{proof}

\begin{proof}[Proof of \autoref{Prop_APrioriBounds}]
  Apply \autoref{Prop_IntegrationByParts} with $f=1$ and $U=M$ to obtain
  \begin{equation*}
    \int_M |\nabla_A \Psi|^2 \leq -\int_M \frac{s}{4}|\Psi|^2 + \<F_B\Psi,\Psi\>
    = O(1).
  \end{equation*}
  Combine this with Kato's inequality and the Sobolev embedding $W^{1,2} \into L^6$ to obtain
  \begin{equation*}
    \|\Psi\|_{L^6}
    = O(1).
  \end{equation*}

  The operator $\Delta+1$ is invertible and has a positive Green's function $G$, which has an expansion of the form
  \begin{equation*}
    G(x,y)
    = \frac{1}{4\pi}\frac{e^{-d(x,y)}}{d(x,y)} + O\(d(x,y)\).
  \end{equation*}
  Apply \autoref{Prop_IntegrationByParts} with $f=G(x,\cdot)$ and $U = M \setminus B_\sigma(x)$, and pass to the limit $\sigma = 0$ to obtain
  \begin{equation*}
    \frac12 |\Psi|^2(x) + \int_M G(x,\cdot) \(\tan(\alpha)^{-2} |\mu(\Psi)|^2 + |\nabla_A\Psi|^2\)
    \lesssim \int_M G(x,\cdot)|\Psi|^2.
  \end{equation*}
  The right-hand side of this equation is $O(1)$ because of the $L^6$--bound on $\Psi$.
  Taking the supremum of the left-hand side over all $x \in M$ yields the desired bounds.
\end{proof}


\section{Curvature controls \texorpdfstring{$\Psi$}{Psi}}
\label{Sec_FControlsPsi}

This section begins the proof of the more difficult second part of \autoref{Thm_A}.

\begin{definition}
  The \emph{critical radius} $\rho(x)$ of a connection $A \in \sA(\sL)$ is defined by
  \begin{equation*}
    \rho(x) := \sup \set{ r \in (0,r_0] : r^{1/2}\|F_A\|_{L^2(B_r(x))} \leq 1 }.
  \end{equation*}
  If the base-point $x$ is obvious from the context and confusion is unlikely to arise, we will often drop $x$ from the notation and just write $\rho$.
\end{definition}

\begin{remark}
  While some constant must be chosen in the definition of $\rho$, the precise choice is immaterial, since we are working with an abelian gauge group $G = \U(1)$.
  In general, $1$ should be replaced by the Uhlenbeck constant of $G$ on $M$.
\end{remark}

\begin{prop}
  \label{Prop_PsiW22Bound}
  Suppose $(A,\Psi) \in \sA(\sL) \times \Gamma\(\Hom(E,\slS\otimes \sL)\)$ satisfies
  \begin{equation*}
    \slD_{A\otimes B} \Psi = 0.
  \end{equation*}
  If $x \in M$ and $\delta \in (0,1]$, then
  \begin{align*}
    r^{1/2}\|\nabla_{A\otimes B}^2 \Psi\|_{L^2(B_{(1-\delta)r}(x))}
    &\lesssim_\delta r^{-3/2}\|\Psi\|_{L^2(B_r(x))} + r^{-1/2}\|\nabla_{A\otimes B}\Psi\|_{L^2(B_r(x))} \\
    &\quad\quad + r^{1/2}\|F_A\|_{L^2(B_r(x))}\|\Psi\|_{L^\infty(B_r(x))}.
  \end{align*}
  In particular, if $(A,\Psi,\alpha) \in \sA(\sL) \times \Gamma\(\Hom(E,\slS\otimes \sL)\) \times [0,\pi/2]$ is a solution of \eqref{Eq_nSW}, then
  \begin{equation*}
    \rho^{1/2}\|\nabla_{A\otimes B}^2 \Psi\|_{L^2(B_{\rho/2}(x))} = O(1).
  \end{equation*}
\end{prop}

\begin{proof}
  The statement is scale-invariant, so we might as well assume that $B_r(x)$ is a geodesic ball $B_1$ of radius one (with an almost flat metric).
  Fix a cut-off function $\chi$ which is supported in $B_{1-\delta/2}$ and is equal to one in $B_{1-\delta}$.
  A straight-forward direct computation using integration by parts yields
  \begin{align*}
    \int |\nabla_{A\otimes B}^2(\chi\Psi)|^2
    &\lesssim
      \int |\nabla_{A\otimes B}^* \nabla_{A\otimes B} (\chi\Psi)|^2  \\
    &\quad\quad
      + |F_{A\otimes B}||\nabla_{A\otimes B} (\chi\Psi)|^2
      + |F_{A\otimes B}||\chi\Psi||\nabla_{A\otimes B}^2(\chi\Psi)|.
  \end{align*}
  Since, as a consequence of the Weitzenböck formula~\eqref{Eq_Weitzenbock},
  \begin{equation*}
    \nabla_{A\otimes B}^* \nabla_{A\otimes B} (\chi\Psi)
    = -\frac{s}{4}\chi\Psi - F_{A\otimes B} (\chi\Psi)
    - 2\nabla^{A\otimes B}_{\nabla \chi} \Psi
    + (\Delta\chi) \Psi,
  \end{equation*}
  we can write
  \begin{equation}
    \label{Eq_D2Psi}
    \begin{split}
      \int |\nabla_{A\otimes B}^2(\chi\Psi)|^2 
      &\lesssim_\delta
      \int |F_{A\otimes B}|^2|\chi\Psi|^2
      + |F_{A\otimes B}||\nabla_{A\otimes B}(\chi\Psi)|^2 \\
      &\quad\quad + |F_{A\otimes B}||\chi\Psi||\nabla_{A\otimes B}^2(\chi\Psi)| \\
      &\quad\quad + |\nabla_{A\otimes B}\Psi|^2 + |\Psi|^2.
    \end{split}
  \end{equation}
  The first and the last two terms are already acceptable.
  The third term is bounded by
  \begin{equation*}
    \epsilon^{-1}\|F_{A\otimes B}\|_{L^2}^2\|\Psi\|_{L^\infty}^2 + \epsilon  \|\nabla_{A\otimes B}^2(\chi\Psi)\|_{L^2}^2
  \end{equation*}
  for all $\epsilon > 0$.
  The first term is acceptable and the second one can be rearranged to the left-hand side of \eqref{Eq_D2Psi} provided $\epsilon$ is chosen sufficiently small.
  The second term can be bounded by
  \begin{equation*}
   \|F_{A\otimes B}\|_{L^2} \|\nabla_{A\otimes B}(\chi\Psi)\|_{L^4}^2.
  \end{equation*}
  Using the Gagliardo--Nirenberg interpolation inequality
  \begin{equation*} 
    \|f\|_{L^4} \lesssim \|\nabla f\|_{L^2}^{3/4}\|f\|_{L^2}^{1/4}
  \end{equation*}
  and Kato's inequality we obtain
  \begin{align*}
    \|\nabla_{A\otimes B}(\chi\Psi)\|_{L^4}^2
    &\lesssim \|\nabla_{A\otimes B}^2(\chi\Psi)\|_{L^2}^{3/2}\|\nabla_{A\otimes B}(\chi\Psi)\|_{L^2}^{1/2} \\
    &\leq \epsilon\|\nabla_{A\otimes B}^2(\chi\Psi)\|_{L^2}^2 + \epsilon^{-3}\|\nabla_{A\otimes B}(\chi\Psi)\|_{L^2}^2
  \end{align*}
  for all $\epsilon > 0$.
  The first term can be rearranged to the left-hand side of \eqref{Eq_D2Psi} provided $\epsilon$ is chosen sufficiently small and the second term is acceptable.
\end{proof}


\section{A frequency function}
\label{Sec_Frequency}

In view of \autoref{Prop_PsiW22Bound} the following result is the key to proving \autoref{Thm_A}.

\begin{prop}
  \label{Prop_NormPsiControlsRho}
  There exists a constant $\omega > 0$ such that for each solution $(A,\Psi,\alpha) \in \sA(\sL) \times \Gamma\(\Hom(E,\slS\otimes \sL)\) \times (0,\pi/2)$ of \eqref{Eq_nSW} we have
  \begin{equation*}
    \rho(x) \gtrsim \min\set{1,|\Psi|^{1/\omega}(x)}.
  \end{equation*}
\end{prop}

The proof of this proposition will be given in \autoref{Sec_NormPsiControlsRho}.
In this section we lay the groundwork by introducing the following tool:

\begin{definition}
  \label{Def_Frequency}
  The \emph{frequency function} $\scN_x\co (0,r_0] \to [0,\infty)$ of $(A,\Psi,\alpha) \in \sA(\sL) \times \Gamma\(\Hom(E,\slS\otimes \sL)\) \times (0,\pi/2)$ at $x \in M$ is defined by
  \begin{equation*}
    \scN_x(r) := \frac{rH_x(r)}{h_x(r)}
  \end{equation*}
  with 
  \begin{equation*}
    \begin{split}
      H_x(r)
      &:= \int_{B_r(x)} |\nabla_{A\otimes B} \Psi|^2 + \tan(\alpha)^{-2}|\mu(\Psi)|^2
      \\ \andq
      h_x(r)
      &:= \int_{\del B_r(x)} |\Psi|^2.
    \end{split}
  \end{equation*}
  If the base-point $x$ is obvious from the context and confusion is unlikely to arise, we will often drop $x$ from the notation and just write $\scN$, $H$ and $h$.
\end{definition}

\begin{remark}
  The notion of frequency function, introduced by Almgren \cite{Almgren1979}, is important in the study of singular/critical sets of elliptic partial differential equations, see, e.g., \cites{Han1998,Naber2014}.
  Our frequency function is an adaptation of the one used by Taubes in \cite{Taubes2012}.
\end{remark}

Throughout the rest of this section we will assume that $(A,\Psi,\alpha) \in \sA(\sL) \times \Gamma\(\Hom(E,\slS\otimes \sL)\) \times (0,\pi/2)$ satisfies \eqref{Eq_nSW} and fix a point $x \in M$.
We establish various important properties of the frequency function.
In particular, we show that:
\begin{itemize}
\item
  $\scN$ is almost monotone increasing in $r$.
\item
  $\scN$ controls the growth of $h$.
\item
  If $|\Psi|(x) > 0$, then $\scN(r)$ goes to zero as $r$ goes to zero.
\end{itemize}
Moreover, we study the base-point dependence of $\scN$.

\subsection{Almost monotonocity of \texorpdfstring{$\scN$}{N}}
\label{Sec_AlmostMonotonicityOfFrequency}

\begin{prop}
  \label{Prop_NDerivativeLowerBound}
  The derivative of the frequency is bounded below as follows
  \begin{equation}
    \label{Eq_NDerivativeLowerBound}
      \scN'(r) \geq O(r)(1+\scN(r)).
  \end{equation}
\end{prop}

Before we embark on the proof, which occupies the remainder of this subsection, let us note the following consequence:

\begin{prop}
  \label{Prop_NMaxGrowth}
  If $0 < s \leq r$, then
  \begin{equation}
    \label{Eq_NMaxGrowth}
    \scN(s) \leq e^{O\(r^2-s^2\)}\scN(r) + O\(r^2-s^2\).
  \end{equation}
\end{prop}

\begin{proof}
  From \eqref{Eq_NDerivativeLowerBound} it follows that
  \begin{equation*}
    \del_r \log(\scN(r)+1) \geq -2cr.
  \end{equation*}
  This integrates to
  \begin{equation*}
    \log(\scN(r)+1) - \log(\scN(s)+1) \geq -c(r^2 - s^2),
  \end{equation*}
  i.e.,
  \begin{equation*}
    \scN(s)+1 \leq e^{c(r^2-s^2)}(\scN(r)+1),
  \end{equation*}
  which directly implies \eqref{Eq_NMaxGrowth}.
\end{proof}

The derivative of the frequency is
\begin{equation}
  \label{Eq_NDerivative}
  \scN'(r) = \frac{H(r)}{h(r)} + \frac{rH'(r)}{h(r)} - \frac{rh'(r)H(r)}{h(r)^2};
\end{equation}
hence, to prove \autoref{Prop_NDerivativeLowerBound} we need to better understand $h'$ and $H'$.
This is what is achieved in the following.

\begin{prop}
  \label{Prop_hDerivative}
  The derivative of $h$ satisfies
  \begin{equation}
    \label{Eq_hDerivative}
    h'(r) = 2h(r)/r + \int_{\del B_r(x)} \del_r|\Psi|^2 + O(r)h(r)
  \end{equation}
  and
  \begin{equation}
    \label{Eq_hDerivative2}
    h'(r) = \(2+2\scN(r)+O(r^2)\)h(r)/r.
  \end{equation}
  Moreover,
  \begin{equation}
    \label{Eq_ErrorBound}
    \int_{B_r(x)} |\Psi|^2 \lesssim rh(r).
  \end{equation}
\end{prop}

\begin{proof}
  We proceed in four steps.

  \setcounter{step}{0}
  \begin{step}
    The identity \eqref{Eq_hDerivative} is clear if the metric is flat near $x$;
    the term $O(r)h(r)$ compensates for the metric possibly being non-flat.
  \end{step}

  \begin{step}\label{Step_ErrorBound0}
    $\int_{B_r(x)} |\Psi|^2 \lesssim (1+\scN(r))rh(r).$
  \end{step}

  Apply the following general fact
  \begin{equation*}
    \int_{B_r(x)} \rd(x,\cdot)^{-2}f^2 \lesssim r^{-1} \int_{\del B_r(x)} f^2 + \int_{B_r(x)} |\rd f|^2,
  \end{equation*}
  which can be proved using integration by parts and Cauchy--Schwarz, to $f=|\Psi|$ and use Kato's inequality.

  \begin{step}
    \label{Step_h'>0}
    $h'(r)>0$.
  \end{step}

  Use \autoref{Prop_IntegrationByParts} with $U=B_r(x)$ and $f=1$ to write
  \begin{equation}
    \label{Eq_HAsBoundaryIntegral}
    \int_{\del B_r(x)} \del_r|\Psi|^2 = 2H(r) + O(1)\int_{B_r(x)} |\Psi|^2.
  \end{equation}
  The estimate from \autoref{Step_ErrorBound0} implies
  \begin{equation*}
    h'(r) = \(1+O(r^2)\)(2+2\scN(r))h(r)/r
  \end{equation*}
  which is non-negative because $r \leq r_0$.

  \begin{step}
    Proof of \eqref{Eq_hDerivative2} and \eqref{Eq_ErrorBound}.
  \end{step}

  The bound \eqref{Eq_ErrorBound} follows directly from $h'(r)>0$.
  Using \eqref{Eq_ErrorBound} in \autoref{Step_h'>0} instead of the estimate from \autoref{Step_ErrorBound0} immediately implies \eqref{Eq_hDerivative2}.
\end{proof}

\begin{prop}
  \label{Prop_HDerivative}
  The derivative of $H$ satisfies
  \begin{multline}
    \label{Eq_HDerivative}
    H'(r)
    = \frac{1}{r} H(r) + \int_{\del B_r(x)} 2|\nabla^{A\otimes B}_r\Psi|^2 + \tan(\alpha)^{-2}|i(\del_r)\mu(\Psi)|^2 \\
    + O\bigl((1+\scN(r))h(r)\bigr).
  \end{multline}
  Here we think of $\mu(\Psi)$ as a $2$--form via \eqref{Eq_Lambda=su}.
\end{prop}

\begin{proof}
  The punctured ball $\dot B_{r_0}(x) := B_{r_0}(x)\setminus \set{x}$ is foliated by the surfaces $\del B_r(x)$ with normal vector field $\del_r$.
  According to \cite{Baer2005}*{Section 3} the restriction of the spin bundle on $\dot B_{r_0}(x)$ to $\del B_r(x)$ can be identified with the spin bundle on $\del B_r(x)$ and if $\tilde\gamma$, $\tilde\nabla$ and $\tilde\slD$ denote the Clifford multiplication, spin connection and Dirac operator on $\del B_r(x)$ respectively, then for $v \in T\del B_r(x)$:
  \begin{align*}
    \gamma(v)
    &= -\gamma(\del_r)\tilde\gamma(v), \\
    \nabla_v
    &= \tilde\nabla_v + \frac{e^{O(r^2)}}{2r}\tilde\gamma(v) \qand \\
    \slD
    &= \gamma(\del_r) (\nabla_r + \frac{e^{O(r^2)}}{r} - \tilde\slD).
  \end{align*}
  (If the metric on $B_{r_0}(x)$ is flat, then the mean curvature of $\del B_r(x)$ is $-\frac{1}{r}$.
  In general, there is a correction term; hence, the term $e^{O(r^2)}$.)
  In particular, $\slD \Psi = 0$ is equivalent to
  \begin{equation*}
    \tilde\slD \Psi = \nabla_r\Psi + \frac{e^{O(r^2)}}{r}\Psi.
  \end{equation*}

  For $\Psi$ a harmonic spinor on $B_r(x)$ we compute:
  \begin{align*}
    \int_{\del B_r(x)} |\nabla\Psi|^2 - |\nabla_r\Psi|^2 
    &= \int_{\del B_r(x)} |\tilde\nabla\Psi + \frac{e^{O(r^2)}}{2r} \tilde\gamma(\cdot)\Psi|^2 \\
    &= \int_{\del B_r(x)} |\tilde\nabla\Psi|^2 - \frac{e^{O(r^2)}}{r} \langle\tilde\slD\Psi,\Psi\rangle + \frac{e^{O(r^2)}}{2r^2} |\Psi|^2 \\
    &= \int_{\del B_r(x)} |\tilde\nabla\Psi|^2 - \frac{e^{O(r^2)}}{r}\<\nabla_r\Psi,\Psi\> - \frac{e^{O(r^2)}}{2r^2} |\Psi|^2.
  \end{align*}
  Using the Weitzenböck formula~\eqref{Eq_Weitzenbock} the first term can be written as
  \begin{align*}
    \int_{\del B_r(x)} |\tilde\nabla\Psi|^2
    &= \int_{\del B_r(x)} \langle\tilde\nabla^*\tilde\nabla\Psi,\Psi\rangle \\
    &= \int_{\del B_r(x)} |\tilde\slD \Psi|^2 - \frac{e^{O(r^2)}}{2r^2} |\Psi|^2 \\
    &= \int_{\del B_r(x)} |\nabla_r\Psi|^2 + \frac{2e^{O(r^2)}}{r} \<\nabla_r\Psi,\Psi\> + \frac{e^{O(r^2)}}{2r^2}|\Psi|^2.
  \end{align*}
  This combined with \eqref{Eq_HAsBoundaryIntegral} and \eqref{Eq_ErrorBound} proves the asserted identity if $A$ and $B$ are product connections.

  If $A$ and $B$ are not the product connection, the computation is identical up to changes in notation \emph{and} in the Weitzenböck formula two additional terms appear.
  The first is
  \begin{equation*}
    -\int_{\del B_r(x)} \<F_A|_{\del B_r(x)},\mu(\Psi)\>
  \end{equation*}
  and the second can be estimated by $O(1)h(r)$.
  If $(e_1,e_2)$ is a local positive orthonormal frame of $T\del B_r(x)$, then the integrand in the above expression is
  \begin{equation*}
    \frac12\<F_A(e_1,e_2)[\tilde\gamma(e_1),\tilde\gamma(e_2)],\mu(\Psi)\>
    = \<F_A(e_1,e_2)\gamma(\del_r),\mu(\Psi)\>.
  \end{equation*}
  To better understand this term, observe that if $\{\cdot,\cdot\}$ denotes the anti-commutator, then
  \begin{equation*}
    \mu(\Psi) = \sum_{m} \frac12\{\mu(\Psi),\gamma_m\}\gamma_m
  \end{equation*}
  and $\<\gamma_m,\gamma_n\> = 2\delta_{mn}$.
  Using $F_A = \tan(\alpha)^{-2}\mu(\Psi)$ we can write the integrand as $\tan(\alpha)^{-2}$ times
  \begin{equation*}
    \frac12 |\{\mu(\Psi),\gamma(\del_r)\}|^2
    = |\mu(\Psi)|^2 - |i(\del_r)\mu(\Psi)|^2.
  \end{equation*}
  This proves \eqref{Eq_HDerivative} in general.
\end{proof}

\begin{proof}[Proof of \autoref{Prop_NDerivativeLowerBound}]
  Plug \eqref{Eq_hDerivative2} and \eqref{Eq_HDerivative} into \eqref{Eq_NDerivative} and use \eqref{Eq_HAsBoundaryIntegral} and \eqref{Eq_ErrorBound} to obtain
  \begin{align*}
    \scN(r)'
    &= \frac{2r}{h(r)} \int_{\del B_r(x)} |\nabla^{A\otimes B}_r \Psi|^2 + \tan(\alpha)^{-2} |i(\del_r)\mu(\Psi)|^2 \\
    &\quad - \frac{2r}{h(r)^2}\(\int_{\del B_r(x)} \<\nabla^{A\otimes B}_r\Psi,\Psi\>\)^2 \\
    &\quad + O(r)\(1+\scN(r)\).
  \end{align*}
  By Cauchy--Schwarz the sum of the first and the third term is positive.
  This completes the proof.
\end{proof}

\subsection{\texorpdfstring{$\scN$}{N} controls the growth of \texorpdfstring{$h$}{h}}
\label{Sec_NControlsGrotwthOfh}

\begin{prop}
  \label{Prop_hIntegralFormula}
  If $0 < s < r$, then
  \begin{equation}
    \label{Eq_hIntegralFormula}
    h(r) = e^{O(r^2)}\(r/s\)^2\exp\(2\int_s^r \scN(t)/t \, \rd t\) h(s).
  \end{equation}
\end{prop}

\begin{proof}
  \autoref{Eq_hDerivative2} can be written as
  \begin{equation*}
    (\log h(r))' = (2+2\scN(r))/r + O(r).
  \end{equation*}
  Integrating this yields \eqref{Eq_hIntegralFormula}.
\end{proof}

\begin{cor}
  \label{Cor_hGrowth}
  If $0 < s < r$, then
  \begin{equation*}
    h(s) \lesssim (s/r)^2 h(r).
  \end{equation*}
  In particular, if $h(s)$ is positive, then so is $h(r)$;
  moreover,
  $|\Psi|^2(x) \lesssim h(r)/r^2$.
\end{cor}

\begin{prop}
  \label{Prop_NControlsGrowthOfh}
  If $0 < s < r$, then
  \begin{equation*}
    e^{O(r^2)}(s/r)^{e^{O(r^2)}(2+2\scN(r))}h(r)
    \leq h(s)
    \leq e^{O(r^2)}(s/r)^{e^{O(r^2)}(2+2\scN(s))}h(r).
  \end{equation*}
\end{prop}

\begin{proof}
  Combine 
  \begin{equation*}
      h(s) = e^{O(r^2)}\(s/r\)^2 \exp\(-2\int_s^r \scN(t)/t \, \rd t\) h(r)
  \end{equation*}
  with
  \begin{align*}
    \int_s^r \scN(t)/t \, \rd t
    &\leq \int_s^r \frac{1}{t}\(e^{O(r^2-t^2)}\scN(r)+O(r^2-t^2)\) \, \rd t \\
    &\leq -\(e^{O(r^2)}\scN(r)+O(r^2)\)\log(s/r)
  \end{align*} and
  \begin{equation*}
    -\(e^{O(r^2)}\scN(s)+O(r^2)\)\log(s/r) \leq \int_s^r \scN(t)/t \, \rd t.
  \end{equation*}
  The last two inequalities are consequences of \autoref{Prop_NMaxGrowth}.
\end{proof}

\subsection{\texorpdfstring{$|\Psi|(x)$}{|Psi|(x)} controls \texorpdfstring{$\scN$}{N}}

\begin{prop}
  \label{Prop_PsiXControlsN}
  If $0 < \omega \ll 1$ and 
  \begin{equation*}
    s \lesssim_\omega \min\{1,|\Psi|^{1/\omega}(x)\}, 
  \end{equation*}
  then $\scN(s) \lesssim \omega$.
\end{prop}

\begin{proof}
  By \autoref{Prop_APrioriBounds}, $h(r) \lesssim r^2$ and, by \autoref{Cor_hGrowth}, $h_x(s)\gtrsim s^2 |\Psi|^2(x)$.
  From \autoref{Prop_NControlsGrowthOfh} it follows that for $s < r$
  \begin{align*}
    (r/s)^{e^{O(r^2)}2\scN(s)+O(r^2)} \leq c^2|\Psi|^{-2}(x);
  \end{align*}
  hence,
  \begin{equation*}
    \scN(s) \lesssim \frac{\log(c|\Psi|^{-1}(x))}{\log(r/s)} + O(r^2).
  \end{equation*}
  If $\sigma := c|\Psi|^{-1}(x) \leq 1$, then the first term is non-positive and setting $r=2\omega$ and $s = \omega$ yields the asserted bound.
  If $\sigma > 1$, set $r = \omega$ and $s = \omega c^{-1/\omega} |\Psi|^{1/\omega}(x) = \omega \sigma^{-1/\omega}$ to obtain
  \begin{equation*}
    \scN(s) 
    \lesssim \omega + O(r^2)
    \lesssim \omega.
    \qedhere
  \end{equation*}
\end{proof}

\subsection{Dependence of \texorpdfstring{$\scN$}{N} on the base-point}

\begin{prop}
  \label{Prop_hDependenceOnBasePoint}
  For $x,y \in M$ and $r > 0$
  \begin{equation*}
    h_x(r) \lesssim \frac{2r+d(x,y)}{r} h_y\bigl(2r+d(x,y)\bigr).
  \end{equation*}
\end{prop}

\begin{proof}
  By \autoref{Cor_hGrowth} and \eqref{Eq_ErrorBound}
  \begin{equation*}
    rh_x(r)
    \lesssim  \int_{B_{2r}(x)} |\Psi|^2
    \leq \int_{B_{2r+d(x,y)}(y)} |\Psi|^2 
    \lesssim \bigl(2r+d(x,y)\bigr) h_y\bigl(2r+d(x,y)\bigr).
    \qedhere
  \end{equation*}
\end{proof}

\begin{prop}
  \label{Prop_FrequencyCannotJumpToMuch}
  Suppose $x \in M$ and $r > 0$ are such that $\scN_x(10r) \leq 1$.
  If $y \in B_{r}(x)$, then $\scN_y(5r) \lesssim \scN_x(10r)$.
\end{prop}

\begin{proof}
  Since
  \begin{equation*}
    \scN_y(5r) = \frac{5r H_y(5r)}{h_y(5r)}
    \lesssim \scN_x(10r) \frac{h_x(10r)}{h_y(5r)},
  \end{equation*}
  it is key to control the latter quotient.
  Using \autoref{Prop_NControlsGrowthOfh} with $\scN_x(10r)\leq 1$ as well as \autoref{Prop_hDependenceOnBasePoint}
  \begin{equation*}
    h_x(10r) \lesssim h_x(r) \lesssim h_y(5r).
    \qedhere
  \end{equation*} 
\end{proof}


\section{\texorpdfstring{$\scN$}{N} controls \texorpdfstring{$\rho(x)$}{rho(x)}}
\label{Sec_NormPsiControlsRho}

In view of \autoref{Prop_PsiXControlsN} it suffices to prove the following in order to complete the proof of \autoref{Prop_NormPsiControlsRho}.

\begin{prop}
  \label{Prop_NControlsRho}
  There are $\omega, \rho_0 > 0$ such that for every solution $(A,\Psi,\alpha) \in \sA(\sL_0) \times \Gamma\(\Hom(E_0, \slS\otimes\sL)\) \times (0,\pi/2)$ of \eqref{Eq_nSW} 
  \begin{equation*}
    \text{if}\quad
    \scN(50r) \leq \omega,
    \quad\text{then}\quad
    \rho \ge \min\{r,\rho_0\}.
  \end{equation*}
\end{prop}

\subsection{Interior \texorpdfstring{$L^2$}{L2}--bounds on the curvature}
\label{Sec_InteriorCurvatureBounds}

We first show that if the critical radius $\rho$ and the frequency $\scN(\rho)$ are very small, then so is the renormalised $L^2$--norm of $F_A$ on $B_{\rho/2}(x)$:

\begin{prop}
  \label{Prop_VerySmallCurvature}
  Let $(A,\Psi,\alpha) \in \sA(\sL_0) \times \Gamma\(\Hom(E_0, \slS\otimes\sL)\) \times (0,\pi/2)$ be a solution of \eqref{Eq_nSW}.
  For any $\epsilon >0$, if 
  \begin{equation*}
    \rho \ll_\epsilon 1 \qandq \scN(\rho) \ll_\epsilon 1,
  \end{equation*}
  then
  \begin{equation*}
    \rho \int_{B_{\rho/2}(x)} |F_A|^2 \leq \epsilon.
  \end{equation*}
\end{prop}

Since
\begin{equation*}
  \frac{\tan(\alpha)^2}{h(\rho)} \leq \(\rho\int_{B_\rho(x)} |F_A|^2\)^{-1} \scN(\rho) = \scN(\rho),
\end{equation*}
this is a direct consequence of the following.

\begin{prop}
  \label{Prop_InteriorCurvatureBounds}
  Denote by $(B_r,g)$ a Riemannian $3$--ball of radius $r > 0$, by $\sL_0$ a $\U(1)$--bundle over $B_r$, by $E_0$ an $\SU(n)$--bundle over $B_r$ and by $B$ a connection on $E_0$.
  Suppose that $(A, \Psi, \alpha) \in \sA(\sL_0) \times \Gamma\(\Hom(E_0, \slS\otimes\sL_0)\) \times (0,\pi/2)$ satisfies \eqref{Eq_nSWNotNormalised}.
  Set
  \begin{gather*}
    \fe:= \frac{r\int_{B_r} |\nabla_{A\otimes B}\Psi|^2}{\int_{\del B_r} |\Psi|^2} + r^2\|R_g\|_{L^\infty(B_r)} + r^2\|F_B\|_{L^\infty(B_r)} \\
    \andq \tau := \frac{\tan(\alpha)}{\sqrt{\int_{\del B_r} |\Psi|^2}}.
  \end{gather*}
  Let $\delta \in (0,1)$ and $\epsilon > 0$.
  If
  \begin{equation*}
    r^{1/2}\|F_A\|_{L^2(B_r)} \leq 1, \quad
    \fe \ll_{\delta,\epsilon} 1 \qandq
    \tau \ll_{\delta,\epsilon} 1,
  \end{equation*}
  then
  \begin{equation*}
    r^{1/2}\|F_A\|_{L^2(B_{(1-\delta)r})} \leq \epsilon.
  \end{equation*}
\end{prop}

The statement of \autoref{Prop_InteriorCurvatureBounds} is invariant under rescaling $B_r$, multiplying $\Psi$ by a constant and changing $\alpha$---hence, $\tan(\alpha)$---accordingly so that \eqref{Eq_nSWNotNormalised} still holds.
Therefore, it suffices to consider the case $r = 1$ and $\int_{\del B_r} |\Psi|^2 = 1$.
Throughout the rest of this subsection assume the hypotheses of \autoref{Prop_InteriorCurvatureBounds} with this normalisation.

\begin{prop}
  \label{Prop_PsiLowerUpperBounds}
  There are constants $0 < \lambda \leq \Lambda = \Lambda(\delta) $ such that in $B_{1-\delta}$
  \begin{equation*}
    |\Psi| \leq \Lambda
    \qandq\text{if $\fe \ll_\delta 1$,}
    \quad\text{then}\quad
    |\Psi| \geq \lambda.
  \end{equation*}
\end{prop}

\begin{proof}
  We proceed in three steps.

  \setcounter{step}{0}
  \begin{step}
    \label{Step_LinftyBoundAwayFromBoundary}
    If $\fe \leq 1$, then for each $x\in B_1$
    \begin{equation*}
      |\Psi|^2(x) \lesssim d(x,\del B_1)^{-2}.
    \end{equation*}
    In particular, $|\Psi| \leq \Lambda(\delta) = O(1/\delta)$.
  \end{step}

  We use a slight modification of the argument used to prove \autoref{Prop_APrioriBounds}.
  It follows from \eqref{Eq_ErrorBound} that $\|\Psi\|_{L^2(B_1)} = O(1)$ and thus $\||\Psi|\|_{W^{1,2}(B_1)} = O(1)$;
  hence, by Kato's inequality and Sobolev embedding we have $\|\Psi\|_{L^6(B_1)} = O(1)$.

  Let $G$ denote the Green's function for $\Delta$ on $B_1$.
  Fix $x \in B_1$ and set $f := G(x,\cdot)$.
  Then
  \begin{equation*}
    f \lesssim \frac{1}{d(x,\cdot)} \qandq
    |\nabla f| \lesssim \frac{1}{d(x,\cdot)^2}. 
  \end{equation*}
  Apply \autoref{Prop_IntegrationByParts} with $f$ as above and $U = B_1 \setminus B_\sigma(x)$, and pass to the limit $\sigma = 0$ to obtain
  \begin{equation*}
    |\Psi|^2(x) \lesssim \int_{B_1} \frac{|\Psi|^2}{d(x,\cdot)} + d(x,\del B_1)^{-1} \int_{\del B_1} \del_r |\Psi|^2 + d(x,\del B_1)^{-2}.
  \end{equation*}
  The first term is $O(1)$ since $\|1/d(x,\cdot)\|_{L^{3/2}(B_1)} = O(1)$.
  Applying \autoref{Prop_IntegrationByParts} again with $f=1$ and $U=B_1$ gives
  \begin{equation*}
    \int_{\del B_1} \del_r |\Psi|^2 \lesssim \int_{B_1} |\Psi|^2 + |\nabla_A \Psi|^2 + \tau^{-2}|\mu(\Psi)|^2 = O(1).
  \end{equation*}
  Here we have also used that 
  \begin{equation}
    \label{Eq_MuL2Bound}
    \|\mu(\Psi)\|_{L^2(B_1)} = \tau^2\|F_A\|_{L^2(B_1)} \leq \tau^2.
  \end{equation}

  \begin{step}
    \label{Step_HolderBound}
    We have $[|\Psi|]_{C^{0,1/4}(B_{1-\delta})} \lesssim_\delta \fe^{1/8}$.
  \end{step}

  Combining the Gagliardo--Nirenberg interpolation inequality
    \begin{equation*} 
      \|f\|_{L^4(B_{1-\delta})} \lesssim_\delta \|\nabla f\|_{L^2(B_{1-\delta})}^{3/4}\|f\|_{L^2(B_{1-\delta})}^{1/4} + \|f\|_{L^2(B_{1-\delta})},
    \end{equation*}
   with Kato's inequality, we obtain 
  \begin{equation}\label{Eq_NablaPsiL4}
  \begin{aligned}
       \|\nabla_{A \otimes B}\Psi\|_{L^4(B_{1-\delta})} 
      &\lesssim_\delta \|\nabla|\nabla_{A\otimes B}\Psi|\|_{L^2(B_{1-\delta})}^{3/4}\|\nabla_{A \otimes B}\Psi\|_{L^2(B_{1-\delta})}^{1/4} \\
      &\quad\quad + \|\nabla_{A \otimes B}\Psi\|_{L^2(B_{1-\delta})} \\
      &\lesssim_\delta \fe^{1/8}.
    \end{aligned} 
  \end{equation} 
  The asserted estimate now follows from Morrey's inequality combined with Kato's inequality. 

  \begin{step}
    There is a constant $\lambda > 0$ such that if $\fe \ll_{\delta} 1$, then in $B_{1-\delta}$
    \begin{equation*}
      |\Psi| \geq \lambda.
    \end{equation*}
  \end{step}

  We know from \autoref{Prop_NControlsGrowthOfh} that
  \begin{align*}
    \int_{\del B_{1-\delta}} |\Psi|^2 \gtrsim \int_{\del B_{1}} |\Psi|^2 = 1,
  \end{align*}
  which proves the lower bound on $|\Psi|$ when combined with \autoref{Step_HolderBound}.
\end{proof}

\begin{prop}
  If $\fe \leq 1$, then
  \begin{equation*}
    \|\mu(\Psi)\|_{L^\infty(B_{1-\delta})} \lesssim_\delta \tau^{1/8}.
  \end{equation*}
\end{prop}

\begin{proof}
  Using Kato's inequality, \autoref{Prop_PsiLowerUpperBounds} and \eqref{Eq_NablaPsiL4} we obtain
  \begin{align*}
    \|\nabla^2|\mu(\Psi)|\|_{L^2(B_{1-\delta})}
    &\lesssim
    \|\nabla_{A\otimes B}^2\Psi\|_{L^2(B_{1-\delta})}\|\Psi\|_{L^\infty(B_{1-\delta})} + \|\nabla_{A\otimes B}\Psi\|_{L^4(B_{1-\delta})}^2 \\
    &\lesssim_\delta 1.
  \end{align*}
  Hence, using the Gagliardo--Nirenberg interpolation inequality
  \begin{equation*}
    \|\nabla f\|_{L^4(B_{1-\delta})} \lesssim \|\nabla^2 f\|_{L^2(B_{1-\delta})}^{7/8}\|f\|_{L^2(B_{1-\delta})}^{1/8} + \|f\|_{L^2(B_{1-\delta})}
  \end{equation*}
  and Morrey's inequality we obtain
  \begin{equation*}
    \||\mu(\Psi)|\|_{C^{1/4}(B_{1-\delta})}
    \lesssim_{\delta} \||\mu(\Psi)|\|_{W^{1,4}(B_{1-\delta})}
    \lesssim_{\delta} \tau^{1/8}.
    \qedhere
  \end{equation*}
\end{proof}

\begin{proof}[Proof of \autoref{Prop_InteriorCurvatureBounds}]
  By a straight-forward calculation
  \begin{equation*}
    \mu(\mu(\Psi)\Psi,\Psi) = \frac12|\Psi|^2\mu(\Psi) + \mu(\Psi)\circ \mu(\Psi) - \frac12 \tr (\mu(\Psi)\circ \mu(\Psi))\,\id_\slS.  
  \end{equation*}
  Using this and the Weitzenböck formula \eqref{Eq_Weitzenbock} we get
  \begin{align*}
    \nabla^*\nabla\mu(\Psi)
    &= 2 \mu(\nabla_{A\otimes B}^*\nabla_{A\otimes B}\Psi,\Psi) - 2 \<\mu(\nabla_{A\otimes B}\Psi,\nabla_{A\otimes B}\Psi)\> \\
    &= -\(\tau^{-2}|\Psi|^2+\frac{s}{2}\)\mu(\Psi) \\
    &\quad\quad+ 2\tau^{-2}\mu(\Psi)\circ \mu(\Psi) - \tau^{-2}\tr(\mu(\Psi)\circ\mu(\Psi))\,\id_\slS \\
    &\quad\quad- 2\mu(F_B\Psi,\Psi) - 2 \<\mu(\nabla_{A\otimes B}\Psi,\nabla_{A\otimes B}\Psi)\>.
  \end{align*}
  where $\<\cdot,\cdot\>$ denotes the contraction $T^*M\otimes T^*M\to \R$.
  
  Fix a cut-off function $\chi$ which is supported in $B_{1-\delta/2}$ and is equal to one in $B_{1-\delta}$.
  Then the above yields
  \begin{align*}
    &\int \chi|\nabla\mu(\Psi)|^2 + \(\tau^{-2}|\Psi|^2+\frac{s}{2}\) \chi|\mu(\Psi)|^2 \\
    &\quad= \int 2\chi\tau^{-2}\<\mu(\Psi)\circ \mu(\Psi),\mu(\Psi)\> - 2\chi\<\mu(F_B\Psi,\Psi),\mu(\Psi)\> \\
    &\quad\quad\quad - 2\chi\<\<\mu(\nabla_{A\otimes B}\Psi,\nabla_{A\otimes B}\Psi)\>,\mu(\Psi)\> \\
    &\quad\quad\quad - \langle\nabla^{A\otimes B}_{\nabla \chi}\mu(\Psi),\mu(\Psi)\rangle
  \end{align*}
  Since $\|\mu(\Psi)\|_{L^\infty(B_{1-\delta/2})} \lesssim_\delta \tau^{1/8}$, the first term on the right hand side can be bounded by
  \begin{equation*}
    c_\delta \tau^{-2+1/8} \int \chi|\mu(\Psi)|^2.
  \end{equation*}
  Thus, using \autoref{Prop_PsiLowerUpperBounds} and~\eqref{Eq_NablaPsiL4}, for $\fe \ll_\delta 1$ and $\tau \ll_\delta 1$, we obtain 
  \begin{align*}
    \int \chi|\mu(\Psi)|^2
    &\lesssim_\delta \tau^2 \int \(|F_B||\Psi|^2+|\nabla_{A\otimes B}\Psi|^2+|\Psi||\nabla_{A\otimes B}\Psi|\)|\mu(\Psi)| \\
    &\lesssim_\delta \tau^4(\fe + \fe^{1/4}).
  \end{align*}
  This implies the assertion because $F_A = \tau^{-2}\mu(\Psi)$.
\end{proof}

\subsection{Proof of \autoref{Prop_NControlsRho}}

If the assertion does not hold, then there exist solutions $(A, \Psi, \alpha) \in  \sA(\sL) \times \Gamma\(\Hom(E,\slS\otimes \sL)\) \times (0,\pi/2]$ of \eqref{Eq_nSW} and $x \in M$ with $\rho \leq \epsilon$ and $\scN(50\rho) \leq \epsilon$ for arbitrarily small $\epsilon > 0$.
The next four steps show that this is impossible.

\setcounter{step}{0}
\begin{step}
  \label{Step_m_nu'}
  There is a point $x' \in B_{2\rho(x)}(x)$ such that
  \begin{equation*}
    \rho(x') \leq \rho(x) \qandq
    \rho(x') \leq 2 \min \set{ \rho(y) : y \in B_{\rho(x')}(x') }.
  \end{equation*}
\end{step}

Construct a sequence $x_k$ inductively. 
Set $x_0 := x$ and assume that $x_k$ has been constructed. 
If 
\begin{equation*}
  \rho(x_k) \leq 2 \min \set{ \rho(y) : y \in B_{\rho(x_k)}(x_k) },
\end{equation*}
then we set $x' := x_k$.
Otherwise we choose $x_{k+1} \in B_{\rho(x_k)}(x_k)$ such that
\begin{equation*}
  \rho(x_{k+1}) < \frac12 \rho (x_k).
\end{equation*}
By construction we have $\rho (x_{k+1})< \frac{1}{2^k}\rho(x)$. 
Since $\rho(\cdot)$ is bounded below for a fixed $(A,\Psi,\alpha)$, this sequence must terminate for some $k$.
Note that 
\begin{equation*}
  d(x,x') \leq \sum_{i=0}^k \rho(x_i) \leq 2\rho(x).
\end{equation*}

\begin{step}
  For each $y \in B_{\rho(x')}(x')$ we have $\rho(y) \lesssim \epsilon$ and $\scN_y(\rho(y)) \lesssim \epsilon$.
\end{step}

If $y \in B_{\rho(x')}(x')$, then $B_{2\rho(x')}(y) \supset B_{\rho(x')}(x')$;
hence,
\begin{equation*}
  \int\limits_{B_{2\rho(x')}(y)}|F_A|^2 \geq \frac{1}{\rho(x')} > \frac{1}{2\rho(x')}
\end{equation*}
and therefore $\rho(y) < 2\rho(x') \leq 2\rho(x) \lesssim \epsilon$.
Since $y \in B_{5\rho(x)}(x)$, we can apply \autoref{Prop_FrequencyCannotJumpToMuch} with $r = 5\rho(x)$ to deduce that $\scN_y(\rho(y))\leq e^{O(\epsilon^2)} \scN_y(25\rho(x)) + O(\epsilon^2) \lesssim \scN_{x}(50\rho(x)) + O(\epsilon^2) \lesssim \epsilon$.

\begin{step}
  There exists a finite set $\set{y_1,\ldots, y_k} \subset B_{\rho(x')}(x')$ with $k = O(1)$ such that
  \begin{equation*}
    \bigcup B_{\rho(y_i)/2}(y_i) \supset B_{\rho(x')}(x').
  \end{equation*}
\end{step}

It follows from the first step that for each $y \in B_{\rho(x')}(x')$ we have $\rho(y) \geq \frac12\rho(x')$.
This implies the existence of a finite set $\set{y_i}$ with the desired properties.

\begin{step}
  We prove the proposition.
\end{step}

By \autoref{Prop_VerySmallCurvature} and the previous steps
\begin{equation*}
  \int\limits_{B_{\rho(y_i)/2}(y_i)} |F_A|^2
  \lesssim \frac{\epsilon}{\rho(y_i)}
  \lesssim \frac{\epsilon}{\rho(x')};
\end{equation*}
hence,
\begin{equation*}
  \int\limits_{B_{\rho(x')}(x')} |F_A|^2 \lesssim \frac{\epsilon}{\rho(x')}.
\end{equation*}
If $\epsilon \ll 1$, this contradicts the definition of $\rho(x')$.
\qed


\section{Convergence on \texorpdfstring{$M\setminus Z$}{M - Z}}
\label{Sec_Convergence}

In this section we prove the following convergence result, which completes the proof of \autoref{Thm_A} (except for the statement regarding the size of $Z$).

\begin{prop}
  \label{Prop_Convergence}
  In the situation of \autoref{Thm_A} if $\limsup \alpha_i = 0$ and with  $|\Psi|$ as in \autoref{Prop_NormConvergence} after passing to a further subsequence the following hold:
  \begin{enumerate}
  \item
    There is a constant $\gamma > 0$ such that $|\Psi_i|$ converges to $|\Psi|$ in $C^{0,\gamma}$.
    In particular, the set $Z := |\Psi|^{-1}(0)$ is closed.
  \item
    There is a flat connection $A$ on $\sL|_{M\setminus Z}$ with monodromy in $\Z_2$ and $\Psi \in \Gamma\(M\setminus Z,\Hom(E, \slS\otimes \sL)\)$ such that $(A,\Psi,0)$ solves \eqref{Eq_nSW}.
    On $M\setminus Z$ up to gauge transformations $A_i$ converges weakly in $W^{1,2}_\loc$ to $A$ and $\Psi_i$ converges weakly in $W^{2,2}_\loc$ to $\Psi$.
  \end{enumerate}
\end{prop}

To prove this we need the following result.

\begin{prop}
  \label{Prop_HolderBoundForPsi}
  There is a constant $\gamma > 0$ such that whenever $(A,\Psi,\alpha) \in \sA(\sL) \times \Gamma\(\Hom(E,\slS\otimes \sL)\) \times (0,\pi/2]$ is a solution of \eqref{Eq_nSW}, then $[|\Psi|]_{C^{0,\gamma}} = O(1)$.
\end{prop}

\begin{proof}
  Let $x \neq y\in M$.
  We need to uniformly control
  \begin{equation*}
    \frac{||\Psi|(x)-|\Psi|(y)|}{d(x,y)^{\gamma}}
  \end{equation*}
  for some $\gamma > 0$.
  Take $\omega > 0$ as in \autoref{Prop_NormPsiControlsRho}.
  Without loss of generality we can assume that $d(x,y) \leq \omega$ and $0 \neq \nu :=|\Psi|(x) \geq |\Psi|(y)$.
  It follows from \autoref{Prop_NormPsiControlsRho} that
  \begin{equation}
    \label{Eq_RhoLowerBoundNu}
    \rho(x) \gtrsim \min\set{1,\nu^{1/\omega}}.
  \end{equation}

  We distinguish two cases.
  \setcounter{case}{0}
  \begin{case}
    \label{Case_SqrtDLeqRho}
    $d(x,y)^{1/2} \leq \rho(x)/2$.
  \end{case}
  By combining \autoref{Prop_PsiW22Bound} with Sobolev embedding, Morrey's inequality with Kato's inequality we obtain
  \begin{equation*}
    \frac{||\Psi|(x)-|\Psi|(y)|}{d(x,y)^{1/2}}
    \lesssim \|\nabla_{A\otimes B}\Psi\|_{L^6(B_{\rho(x)/2})}
    \lesssim \rho(x)^{-1/2} \lesssim d(x,y)^{-1/4};
  \end{equation*}
  hence,
  \begin{equation*}
    \frac{||\Psi|(x)-|\Psi|(y)|}{d(x,y)^{1/4}} = O(1).
  \end{equation*}

  \begin{case}
    $d(x,y)^{1/2} > \rho(x)/2$.
  \end{case}

  If $\nu \geq 1$, then by \eqref{Eq_RhoLowerBoundNu} we are in \autoref{Case_SqrtDLeqRho}.
  Thus $\nu < 1$ and it follows from \eqref{Eq_RhoLowerBoundNu} that
  \begin{equation*}
    |\Psi|(y) \leq |\Psi|(x) \lesssim \rho(x)^\omega \lesssim d(x,y)^{\omega/2};
  \end{equation*}  
  hence,
  \begin{equation*}
    \frac{||\Psi|(y) - |\Psi|(x)|}{d(x,y)^{\omega/2}} = O(1).
  \end{equation*}

  This proves the proposition with $\gamma := \min\set{\frac14,\frac\omega 2}$.
\end{proof}

\begin{proof}[Proof of \autoref{Prop_Convergence}]  
  \autoref{Prop_HolderBoundForPsi} immediately implies the first part of the proposition.
  We prove the second part.
  If $x \in M\setminus Z$, then, by \autoref{Prop_NormPsiControlsRho}, after passing to a subsequence the critical radius $\rho_{i}(x)$ of $(A_i,\Psi_i,\alpha_i)$ is bounded below by a constant, say, $2R > 0$ depending only on $|\Psi|(x)$.
  By \autoref{Prop_HolderBoundForPsi} we can also assume that $|\Psi_i|$ is bounded away from zero on $B_{2R}(x)$, after possibly making $R$ smaller.
  Combining
  \autoref{Prop_InteriorCurvatureBounds} and \autoref{Prop_PsiXControlsN} yields $L^2$--bounds on $F_{A_i}$ on balls covering $B_R(x)$; hence, by \autoref{Prop_PsiW22Bound}, $W^{2,2}_{A_i}$--bounds on $\Psi_i$.
  After putting $A_i$ in Uhlenbeck gauge on $B_R(x)$ and passing to a subsequence the sequence $(A_i, \Psi_i)$ converges weakly in $W^{1,2} \oplus W^{2,2}$ to a limit $(A, \Psi)$.
  The pair $(A, \Psi)$ satisfies
  \begin{equation*}
    \slD_{A\otimes B}\Psi = 0 \qandq \mu(\Psi) = 0.
  \end{equation*}

  The local gauge transformations can be patched to obtain a global gauge transformation on $M\setminus Z$, see~\cite{Donaldson1990}*{Section 4.2.2}.

  The fact that $A$ has monodromy in $\Z_2$ follows from the discussion in \autoref{App_Fueter}.
\end{proof}


\section{Z is nowhere-dense}
\label{Sec_LimitAnalysis}

Since $\int_M |\Psi|^2 = 1$, we know that $Z$ cannot be the entire space.
To obtain more precise information on $Z$ it turns out to be helpful to apply the ideas from \autoref{Sec_Frequency} to the limit $(A,\Psi)$.
Fix $x \in M$ and define functions $H,h \co [0,r_0] \to [0,\infty)$ by 
\begin{align*}
  H(r) &:= \int_{B_r(x)} |\nabla_{A\otimes B} \Psi|^2 \qand \\
  h(r) &:= \int_{\del B_r(x)} |\Psi|^2.
\end{align*}
Here we extend $|\nabla_{A\otimes B} \Psi|$ by defining it to be zero on $Z$.
If $h(r) > 0$, define
\begin{equation*}
  \scN(r):=\frac{rH(r)}{h(r)}.
\end{equation*}

\begin{prop}
  \label{Prop_UniformConvergenceOfhiHi}
  Denote by $h_i$, $H_i$ the $(A_i,\Psi_i,\alpha_i)$ version of $h$ and $H$ defined in \autoref{Def_Frequency}.
  The sequences of functions $h_i$ and $H_i$ converge uniformly to $h$ and $H$, respectively.
  In particular, $\scN_i(r) \to \scN(r)$ whenever $h(r) > 0$.
\end{prop}

Let us first explain how this implies the following.

\begin{prop}
  \label{Prop_ZIsNowhereDense}
  $Z$ is nowhere-dense.
\end{prop}

\begin{proof}
  Choose $R \geq 0$ as large as possible, but so that $B_R(x) \subset Z$.
  We know that $R$ is finite, because $Z$ is compact.
  By replacing $x$ with a point close to the boundary of $B_R(x)$ we can assume that $R \ll 1$.
  By construction of $R$ there is an $\epsilon \ll 1$ such that $h(R+\epsilon) > 0$.
  In particular, $\scN(R+\epsilon)$ is defined.
  It follows from \autoref{Prop_NControlsGrowthOfh} and \autoref{Prop_UniformConvergenceOfhiHi} that $R=0$.
\end{proof}

\begin{proof}[Proof of \autoref{Prop_UniformConvergenceOfhiHi}]
  That $h_i$ converges uniformly to $h$ is a direct consequence of the $C^{0,\gamma}$ convergence of $|\Psi_i|$.
  The proof of the corresponding statement for $H_i$ has three steps.

  \setcounter{step}{0}
  \begin{step}
    \label{Step_UniformConvergenceofHepsiloni}
    For $\epsilon \in (0,1/2]$ set $Z_\epsilon := |\Psi|^{-1}([0,\epsilon])$.
    The sequence of functions
    \begin{equation*}
      H_{\epsilon,i}(r) := \int_{B_r(x) \setminus Z_\epsilon} |\nabla_{A_i\otimes B} \Psi_i|^2 + \tan(\alpha_i)^{-2} |\mu(\Psi_i)|^2
    \end{equation*}
    converges uniformly to
    \begin{equation*}
      H_{\epsilon}(r) := \int_{B_r \setminus Z_\epsilon} |\nabla_{A\otimes B} \Psi|^2.
    \end{equation*}
  \end{step}

  This follows from the facts that $\tan(\alpha_i)^{-1}\mu(\Psi_i) = \tan(\alpha_i)F_{A_i}$ converges to zero in $L^2(M \setminus Z_\epsilon)$ and $\nabla_{A_i\otimes B}\Psi_i$ converges to $\nabla_{A\otimes B} \Psi$ in $L^2(M \setminus Z_\epsilon)$, see \autoref{Prop_Convergence}.

  \begin{step}
    \label{Step_HiMinusHepsiloni}
    There exists a $\lambda > 0$ such that
    \begin{equation*}
    \int_{Z_\epsilon} |\nabla_{A_i\otimes B} \Psi_i|^2 + \tan(\alpha_i)^{-2}|\mu(\Psi_i)|^2 = O(\epsilon^\lambda).
  \end{equation*}
  \end{step}

  Fix a cut-off function $\chi\co \R \to [0,1]$ with $\chi(t)=1$ for $t\leq 1$ and $\chi(t)=0$ for $t\geq 2$.
  Applying \autoref{Prop_IntegrationByParts} with $f=\chi(\epsilon^{-1}|\Psi_i|)$ and $U=M$, integrating the resulting term with $\Delta|\Psi|$ by parts once and using Kato's inequality yields
  \begin{align*}
    \int_{Z_\epsilon} |\nabla_{A_i\otimes B}\Psi_i|^2 + \tan(\alpha_i)^{-2}|\mu(\Psi_i)|^2 \leq c\epsilon^2 + c \int_{Z_{2\epsilon}\setminus Z_{\epsilon}} |\nabla_{A_i\otimes B}\Psi_i|^2.
  \end{align*}
  Denoting
  \begin{equation*}
    f(\epsilon) := \int_{Z_\epsilon} |\nabla_{A_i\otimes B}\Psi_i|^2 + \tan(\alpha_i)^{-2}|\mu(\Psi_i)|^2
  \end{equation*}
  this can be written as
  \begin{equation*}
    f(\epsilon) \leq \sigma (\epsilon^2 + f(2\epsilon))
  \end{equation*}
  with $\sigma := c/(1+c)$.
  Since $f$ is bounded above and we can assume that $\sigma \geq 1/2$,
  \begin{align*}
    f(\epsilon)
    &\leq \sigma\epsilon^2 \sum_{i=0}^{k-1} (4\sigma)^i + \sigma^kf(2^k\epsilon) \\
    &\leq \epsilon^2 \sigma\(\frac{(4\sigma)^{k-1}-1}{4\sigma-1}\) + c\sigma^k \\
    &\lesssim \epsilon^2(4\sigma)^k + \sigma^k.
  \end{align*}
  With $k := \floor{-\log \epsilon/\log 2}$ this gives
  \begin{align*}
    f(\epsilon) \lesssim \epsilon^{2-\log(4\sigma)/\log 2} + \epsilon^{-\log\sigma/\log 2} \lesssim \epsilon^\lambda
  \end{align*}
  for some $\lambda>0$ depending on $\sigma$ only, since $\log(4\sigma)/\log 2 < 2$.

  \begin{step}
  The sequence of functions $H_i$ converges uniformly to $H$. 
  \end{step}

  Both $|H_{\epsilon}(r)-H(r)|$ and $|H_{\epsilon,i}(r)-H_i(r)|$ converge uniformly to zero as $\epsilon$ goes to zero, the former by monotone convergence and the latter by \autoref{Step_HiMinusHepsiloni};
  hence, the desired convergence follows immediately from \autoref{Step_UniformConvergenceofHepsiloni}.
\end{proof}

\appendix
\appendix
\section{Fueter sections of bundles of moduli spaces of ASD instantons}
\label{App_Fueter}

Recall from \cite{Donaldson1990}*{Section 3.3} that if $E$ denotes a Hermitian vector space of dimension $n$ with fixed determinant, $\slS^+$ denotes the positive spin representation of $\Spin(4)$ and $\sL$ is a Hermitian vector space of dimension one, then
\begin{equation*}
  \bigl(\Hom(E,\slS^+\otimes \sL) \setminus \set{0}\bigr)\!\hkred \U(1)
  = \mathring M_{1,n}
\end{equation*}
the moduli space of centred framed charge one $\SU(n)$ ASD instantons on $\R^4$.

In the situation of \autoref{Sec_Introduction} we have bundles of the above data (which we denote by the same letters) and can construct the bundle
\begin{equation*}
  \fM := (\fs \times \SU(E)) \times_{\Spin(3)\times \SU(n)} \mathring M_{1,n}.
\end{equation*}
Here $\SU(E)$ is the principal $\SU(n)$--bundle of oriented orthonormal frames of $E$ and $\Spin(3)$ acts via the inclusion of the first factor in $\Spin(4) = \Spin_+(3)\times\Spin_-(3)$.
Using the connections on $\fs$ and $E$ we can associate to every section $\fI \in \Gamma(\fM)$ its covariant derivative $\nabla \fI \in \Omega^1(\fI^*V\fM)$.
Here $V\fM := (\fs \times \SU(E)) \times_{\Spin(3)\times \SU(n)} T\mathring M_{1,n}$ is the vertical tangent bundle of $\fM$.
Moreover, there is a Clifford multiplication $\gamma\colon TM\otimes\fI^*V\fM\to \fI^*V\fM$.
Therefore, there is a natural non-linear Dirac operator $\fF$, called the \emph{Fueter operator}, which assigns to a section $\fI \in \Gamma(\fM)$ the vertical vector field
\begin{equation*}
  \fF\fI := \sum_{i=1}^3 \gamma(e_i) \nabla_{e_i}\fI \in \Gamma(\fI^*V\fM).
\end{equation*}

\begin{prop}
  \label{Prop_nSW0Fueter}
  If $A \in \sA(\sL)$ and $\Psi \in \Gamma(M, \Hom(E,\slS \otimes \sL))$ is a solution of
  \begin{equation}
    \label{Eq_nSW0}
    \begin{split}
      \slD_{A\otimes B} \Psi
      &=
        0 \qand \\
      \mu(\Psi)
      &=
        0.
    \end{split}
  \end{equation}
  and $\Psi$ vanishes nowhere, then the induced section $\fI \in \Gamma(\fM)$ solves $\fF\fI = 0$.
  Conversely, each Fueter section $\fI \in \Gamma(\fM)$ lifts to a solution $(A,\Psi)$ of \eqref{Eq_nSW0} for some $\sL$.
\end{prop}

The proof is essentially the same as that of~\cite{Haydys2011}*{Proposition 4.1}.
It is worthwhile to explain how $\sL$ and $A$ are recovered from $\fI$:
the $\U(1)$--bundle $\mu^{-1}(0) \to \mathring M_{1,n}$ has a canonical connection given by orthogonal projection along the $\U(1)$--orbits;
hence, the $\U(1)$--bundle $\underline\sL := (\fs\times \SU (E)) \times_{\Spin(3)\times \SU(n)} \mu^{-1}(0) \to \fM$ inherits a connection $\underline A$; and, finally, $\sL$ and $A$ are obtained via pullback:
\begin{equation*}
  \sL = \fI^*\underline\sL \qandq A = \fI^*\underline A.
\end{equation*}

The following gives more information about $A$.

\begin{prop}
  \label{Prop_nSW0FueterConnection}
  Let $A \in \sA(\sL)$ and $\Psi \in \Gamma(M, \Hom(E,\slS \otimes \sL))$ be a solution of \eqref{Eq_nSW0}.
  Denote $Z := \Psi^{-1}(0)$.
  In this situation the following hold true:
  \begin{enumerate}[(1)]
  \item
    \label{It_Rank2}
    $F := \coim(\Psi) = \ker (\Psi)^\perp$ is a rank $2$ subbundle of $E|_{M\setminus Z}$;
  \item
    \label{It_Sqrt}
    The bundle $\sK := \det F$ has a square root $\sqrt{\sK}$.
    In particular, $\mathring{F} := F \otimes \sK^{-1/2}$ is an $\SU(2)$--bundle.
  \item
    \label{It_nSW0}
    The connection induced on $\sL|_{M\setminus Z} \otimes \sK^{1/2}$ and the induced section $\Phi \in \Gamma(M\setminus Z, \Hom(\mathring{F},\slS \otimes \sL \otimes \sK^{1/2}))$ satisfy \eqref{Eq_nSW0} over $M\setminus Z$ with respect to the induced connection on $\mathring{F}$.
    We have $|\Phi|=|\Psi|$ over $M\setminus Z$ and, hence, $|\Phi|$ extends as a continuous function over $M$ and $Z = |\Phi|^{-1}(0)$.
  \item
    \label{It_Flat}
    The induced connection on $\sL \otimes \sK^{1/2}$ is flat and has $\Z_2$--monodromy.
  \end{enumerate}
\end{prop}

\begin{proof}
  For each $x \in M$,
  $\mu^{-1}(0)\setminus \set{0} \subset \Hom(E,\slS \otimes \sL))_x$ is acted upon transitively by $\R_+ \times \U(E_x)$.
  In particular, it can be checked directly that for one (hence for all) non-zero $\Psi \in \mu^{-1}(0)$ we have $\rk \Psi = 2$.

  The induced section $\Phi \in \Gamma(M, \Hom(F,\slS \otimes \sL))$ defines an isomorphism $F \iso \slS \otimes \sL$,
  hence $\det F \iso \det(\slS \otimes \sL) \iso \sL^{\otimes 2}$.
  This implies \eqref{It_Sqrt}.

  The assertion made in \eqref{It_nSW0} is a consequence of $(A,\Psi)$ satisfying \eqref{Eq_nSW0}.  
  Thus we are left with proving \eqref{It_Flat} in the case $n = 2$.
  In this case, $F = E$, $\sK$ is trivial and $\Phi = \Psi$.
  To see that $A$ is flat with monodromy in $\Z_2$ note that the same is true for the canonical connection on $\mu^{-1}(0) \to \mathring M_{1,2}$:
  note that $\R_+\times \U(2)$ acts transitively on $\mu^{-1}(0)$, and the horizontal distribution is preserved by $\R_+ \times \SU(2)$ and therefore integrable, i.e., the canonical connection is flat.
  Since $\pi_1(\mathring M_{1,2}) = \Z_2$, the monodromy of the canonical connection lies in $\Z_2$. 
\end{proof}

\begin{remark}
  If $\sL \otimes \sK^{1/2}$ carries a flat connection with monodromy in $\Z_2$, then it must be the complexification of a real line bundle $\fl$.
  Solutions to \eqref{Eq_nSW} with $\Spin$--structure $\fs$ and $\U(1)$--bundle $\sL$ are in one-to-one correspondence with solutions with $\Spin$--structure $\fs \otimes \fl$ and $\U(1)$--bundle $\sL \otimes (\fl\otimes\C)$.
  Therefore we can assign to each Fueter section $\fI$ the unique $\Spin$--structure $\fs$ which makes $\sL \otimes \sK^{1/2}$ trivial.
\end{remark}


\bibliography{refs}
\end{document}